\documentclass{amsart}

\usepackage{amsmath,amssymb,amsthm}
\usepackage{color}
\newtheorem{prop}{Proposition}[section]
\newtheorem{thm}[prop]{Theorem}

\newtheorem{cor}[prop]{Corollary}

\theoremstyle{definition}

\newtheorem{defn}[prop]{Definition}
\newtheorem{ex}[prop]{Example}
\newtheorem{rem}[prop]{Remark}

\newtheorem*{ack}{Acknowledgement}

\def\co{\colon\thinspace}

\newcommand{\C}{\mathbb C}

\newcommand{\rmd}{\mathrm d}

\newcommand{\rme}{\mathrm e}

\newcommand{\rmi}{\mathrm i}

\newcommand{\N}{\mathbb N}

\newcommand{\R}{\mathbb R}

\newcommand{\Z}{\mathbb Z}

\newcommand{\lra}{\longrightarrow}
\newcommand{\ra}{\rightarrow}

\newcommand{\wst}{\rmd\mathbf{x}\wedge\rmd\mathbf{y}}
\newcommand{\lamst}{\tfrac12(\mathbf{x}\,\rmd\mathbf{y}-\mathbf{y}\,\rmd\mathbf{x})}

\DeclareMathOperator{\can}{\mathrm{can}}

\DeclareMathOperator{\dist}{\mathrm{dist}}

\DeclareMathOperator{\eh}{\mathrm{EH}}

\DeclareMathOperator{\hz}{\mathrm{HZ}}

\DeclareMathOperator{\im}{\mathrm{im}}

\DeclareMathOperator{\length}{\mathrm{length}}
\DeclareMathOperator{\loc}{\mathrm{loc}}

\DeclareMathOperator{\vol}{\mathrm{vol}}

\begin{document}

\author[Kai Zehmisch]{Kai Zehmisch}
\address{Mathematisches Institut, Universit\"at zu K\"oln, Weyertal 86--90, 50931 K\"oln,Germany}
\email{kai.zehmisch@math.uni-koeln.de}

\title[The codisc radius capacity]{The codisc radius capacity}

\date{23.09.2013}

\begin{abstract}
  We prove a generalization of Gromov's packing inequality
  to symplectic embeddings of the boundaries of two
  balls such that the bounded components of the complements of the
  image spheres are disjoint.
  Moreover,
  we define a capacity
  which measures the size of Weinstein tubular neighbourhoods
  of Lagrangian submanifolds.
  In symplectic vector spaces this leads to bounds
  on the codisc radius for any closed Lagrangian submanifold
  in terms of Viterbo's isoperimetric inequality.
  Furthermore,
  we introduce the spherical variant of the relative Gromov radius
  and prove its finiteness
  for monotone Lagrangian tori in symplectic vector spaces.
\end{abstract}

\subjclass[2010]{53D35, 57R17, 53C22}
\thanks{The author is partially supported by DFG grant
ZE 992/1-1.}

\maketitle



\section{Introduction\label{intro}}

A symplectic manifold $(V,\omega)$
is a smooth $2n$-dimensional manifold $V$
together with a non-degenerate closed $2$-form $\omega$.
The most important examples of symplectic manifolds
are $\mathbb{R}^{2n}$ with
\[
\wst
:=
\rmd x_1\wedge\rmd y_1+\ldots+\rmd x_n\wedge\rmd y_n,
\]
the total space of
cotangent bundles
$\pi\co T^*Q\ra Q$
of smooth manifolds $Q$ with
$\rmd\lambda_{\can}$,
where
\[(\lambda_{\can})_u=u\circ T\pi\]
is the Liouville $1$-form
on $T^*Q$,
and all K\"ahler manifolds with its K\"ahler form.
A symplectomorphism is a diffeomorphism $\varphi$
that preserves the symplectic form $\omega$
in the sense that
$\varphi^*\omega=\omega$.
Symplectomorphisms can be obtained
from compactly supported Hamiltonian
functions $H$ on $V$
as the time-$1$-map of the Hamiltonian
vector field $X_H$
which is defined via
$i_{X_H}\omega=-\mathrm{d}H$.
The canonical lift of diffeomorphisms on a
smooth manifold $Q$ to $T^*Q$
gives another class of examples.

The first observation,
which is due to Liouville,
is that a symplectomorphism preserves
the symplectic volume
$\frac{1}{n!}\omega^n$
of a symplectic manifold $(V,\omega)$.
A natural question is
to what extend a
symplectormorphism is more special
than a volume
preserving diffeomorphism.
In \cite{grom85} Gromov gave the following answer.
Consider the cylinder
\[Z_R=\{x_1^2+y_1^2<R\}\times\mathbb{R}^{2n-2}\]
of radius $R$.
Observe that for any radius $r$
there exists a
linear volume preserving diffeomorphism
of $\R^{2n}$
that maps the open ball $B_r$
of radius $r$ into $Z_R$.
But if $B_r$
embeds into $Z_R$
symplectically
Gromov's non-squeezing theorem
implies $r\leq R$,
see \cite[0.3.A]{grom85}.
More generally,
if only a neighbourhood of
the sphere $S^{2n-1}_r=\partial B_r$
in $\R^{2n}$ embeds symplectically
into $Z_R$
it was show in
\cite{geizeh12,swozil12,swozil12b,zehzil12}
that $r<R$ holds.
We call 
the restriction of a symplectic embedding
of a neighbourhood of $S^{2n-1}_r\subset\R^{2n}$
a {\bf symplectic embedding}
of the sphere $S^{2n-1}_r$.

In \cite[0.3.B]{grom85} Gromov proved
that if two open balls $B_{r_1}$ and $B_{r_2}$
embed symplectically into $B_R$
with disjoint images
then the packing inquality
$r_1^2+r_2^2\leq R^2$
holds.
Ziltener asked whether a packing
obstruction holds for symplectic embeddings
of spheres $S_{r_1}^{2n-1}$ and $S_{r_1}^{2n-1}$
such that bounded components
of the complements of the
images are disjoint.
If $n=2$
this follows
from the packing inequality
for balls
because the existence of
embeddings of
$S_{r_1}^3\sqcup S_{r_2}^3$
into $\R^4$
imply the existence of
embeddings of
$B_{r_1}^4\sqcup B_{r_2}^4$,
see \cite[0.3.C]{grom85}.
In Section \ref{applandthm11}
we will prove the following:

\begin{thm}
  \label{spherespacing}
  Let $n\geq2$.
  Assume that $S_{r_1}^{2n-1}\sqcup S_{r_2}^{2n-1}$
  embeds symplectically into $B_R\subset\R^{2n}$
  such that the bounded components
  of the complements of the images are disjoint.
  Then $r_1^2+r_2^2<R^2$.
\end{thm}

Symplectic embeddings
of balls $B_r\subset\R^{2n}$
into symplectic manifolds $(V,\omega)$
exist by Darboux's theorem.
Any point has a neighbourhood
symplectomorphic to $B_r$
provided $r>0$ is sufficiently small,
cf.\ \cite{hoze94}.
A similar statement is due to
Weinstein \cite{wein71}.
Consider a 
closed Lagrangian submanifold $L$
of $(V,\omega)$,
which is a submanifold
such that $\omega|_{TL}=0$.
A neighbourhood of
the zero section $L$
in the cotangent bundle $T^*L$
is symplectomorphic
to a neighbourhood
of $L$ in $(V,\omega)$.
A natural question is
how large such a neighbourhood can be.

Consider for example
a closed Lagrangian submanifold
$L$ of $\R^{2n}$.
We provide $L$
with the metric $g$ induced
from $\R^{2n}$.
Denote by $D_r^*L$ the
$r$-{\bf codisc bundle} of $L$,
which is
the subset of $T^*L$
consisting of
all covectors of length at most $r$.
By Weinstein's neighbourhood theorem \cite{wein71}
$D_r^*L$ embeds symplectically into $\R^{2n}$
for $r>0$ sufficiently small
identifying the zero section of
$T^*L$ with $L$.
We estimate the largest radius $r$
such that the $r$-codisc bundle
embeds symplectically
in terms of the volume $\vol(L)$
and the
{\bf length of the shortest
  non-trivial closed geodesic}
$\inf(g)$,
which is bounded from below
by the injectivity radius of the metric $g$.
The estimate is obtained
using Viterbo's isoperimetric
inequality \cite{vit00},
see Section \ref{appltothm12}.

\begin{thm}
  \label{isopermest}
  There exists a positive constant $\rho_n$
  such that for any symplectic embedding of $D_r^*L$ into $\R^{2n}$
  the codisc radius $r$ satisfies
  \[
  r^n\leq\rho_n
  \frac{\vol(L)^2}{\inf(g)^n}.
  \]
\end{thm}

The {\bf Gromov radius} is defined by
\[
c_B(V,\omega)=\sup\{\pi r^2\,|\,
\text{$\exists$ symplectic embedding
$B_r\hookrightarrow (V,\omega)$}\},
\]
see \cite{grom85}.
Due to
Gromov's non-squeezing theorem
the quantity $c=c_B$
defines a
{\bf symplectic capacity},
i.e.\
according to Ekeland-Hofer \cite{ekho90},
$c$ satisfies the following conditions:
(Monotonicity)
$c(V,\omega)\leq c(V',\omega')$
provided that there
exists a symplectic embedding
$(V,\omega)\hookrightarrow (V',\omega')$.
(Conformality)
For any positive real number $a$
we have
$c(V,a\omega)=a\, c(V,\omega)$.
(Normalization)
The capacity of the open unit ball $B$
and the open unit cylinder $Z$
are equal to $c(B)=\pi =c(Z)$.
Similarly,
for all symplectic manifolds
$(V,\omega)$ of dimension $\geq4$
the {\bf spherical capacity}
\[
 s(V,\omega):=\sup\{\pi r^2\,|\,
\text{$\exists$ symplectic embedding
  $S^{2n-1}_r\hookrightarrow (V,\omega)$}\}
\]
is a symplectic capacity.
This follows from
the spherical non-squeezing theorem,
see
\cite{geizeh12,swozil12,swozil12b,zehzil12}.

Given a symplectic manifold $(V,\omega)$
a {\bf special capacity}
on subsets $U\subset V$
is a real number
$c(U,\omega)\in [0,\infty]$
satisfying the above
conformality condition
and:
(Non-triviality)
$c(B)$ is positive
and $c(Z)$ is finite;
(Relative monotonicity)
If there
exists a symplectomorphism
of $(V,\omega)$
which maps $U_1$ into $U_2$
then $c(U_1,\omega)\leq c(U_2,\omega)$.

We introduce a quantity
that measures the size
of symplectic neighbourhoods
of Lagrangian submanifolds.
For subsets $U$ in $V$
the {\bf codisc radius capacity}
is defined by
\[
c_D(U)
:=\sup\big\{r\inf(g)\,\big|\,
\text{$(L,g)\subset U$ and $r>0$}\big\},
\]
where the supremum
is taken over all
closed Lagrangian submanifolds
$L\subset U$
of $(V,\omega)$,
over all Riemannian metrics $g$ on $L$,
and over all $r>0$
such that $D^*_r(g)L$
embeds symplectically
into $(V,\omega)$
mapping the zero section
onto $L$.
By Weinstein's neighbourhood theorem
\cite{wein71}
the codisc radius capacity
$c_D(U)$ is positive
for all open subsets
$U$ and all closed
Lagrangian submanifolds $U=L$
of $(V,\omega)$.
In Section \ref{cora}
we will prove the following:

\begin{thm}
  \label{codiscradcap}
  The codisc radius capacity $c_D$
  is a special capacity
  such that
  \[
  c_D(Z)=\pi,\qquad
  c_D(B)\geq\frac{\pi}{n},\qquad\text{and}\qquad
  c_D(P)=\pi,
  \]
  where
  $Z=Z_1$
  is the open unit cylinder,
  $B=B_1$
  is the open unit ball,
  and
  \[
  P=\big\{x_1^2+y_1^2<1,\ldots,x_n^2+y_n^2<1\big\}
  \]
  is the open unit polydisc.
\end{thm}

A more sensible method
to measure the size
of a symplectic neighbourhood
of a closed Lagrangian submanifold $L$
of $\R^{2n}$
was introduced by
Barraud, Biran and Cornea
\cite{barcor06,barcor07,bircor09,bircor09b}.
A symplectic embedding
of the open ball $B_r\subset\R^{2n}$
of radius $r$ into $\R^{2n}$
is called to be
{\bf relative to} $L$
if the real part
\[
D_r=B_r\cap\R^n
\]
of the ball is mapped to $L$
and if the complement $B_r\setminus D_r$
is mapped to $\R^{2n}\setminus L$.
The {\bf relative Gromov radius} is defined by
\[
c_B(L)=\sup\{\pi r^2\,|\,
\text{$\exists$ relative symplectic embedding
$(B_r,D_r)\hookrightarrow (\R^{2n},L)$}\},
\]
see \cite{barcor06,barcor07,bircor09,bircor09b}.
A Lagrangian submanifold $L$
of $\R^{2n}$ is called {\bf monotone}
if the Liouville class of $L$
and the Maslov class of $L$
are positively proportional.
Finiteness of $c_B(L)$
follows
for monotone Lagrangian tori $L$
with \cite[Theorem 1.2.2]{bircor09b},
cf.\ \cite{buh10,damian12}.
As pointed out by McDuff
\cite[p.\ 125]{mcd09} 
it is not known whether
the relative Gromov radius
is finite in general.

Similar to the spherical
capacity \cite{zehzil12}
we consider symplectic
embeddings of neighbourhoods
$U\subset\R^{2n}$
of $S^{2n-1}_r=\partial B_r$
into $\R^{2n}$
such that the induced neighbourhood
$U\cap\R^n$ of the
equatorial sphere
$S^{n-1}_r=\partial D_r$
is mapped to $L$.
We call those
embeddings to be
{\bf relative to} $L$
if no point from
$U\setminus\R^n$
is mapped to $L$.
The
{\bf relative spherical Gromov radius}
is defined by
\[
s(L)=\sup\{\pi r^2\,|\,
\text{$\exists$ relative symplectic embedding
$(S^{2n-1}_r,S^{n-1}_r)\hookrightarrow (\R^{2n},L)$}\}.
\]
Notice that $c_B(L)\leq s(L)$.
We show finiteness
of $s(L)$
for monotone Lagrangian tori.
The proof is given in Section \ref{prfofthm14}.

\begin{thm}
  \label{finofsphervar}
  Let $2n\geq4$.
  The relative spherical Gromov radius is finite
  for any monotone Lagrangian torus in $\R^{2n}$.
\end{thm}


\section{Packing with empty balls\label{pack}}

The aim of this section
is to prove Theorem \ref{spherespacing}.


\subsection{Stretching the neck\label{neck}}

We equip $\R^{2n}$ with the symplectic form $\wst$.
A closed hypersurface $M$ is of
{\bf restricted contact type} in $\R^{2n}$
provided there exists a primitive $1$-form of $\wst$
that restricts to a contact form $\alpha$ on $M$.
In particular $(M,\alpha)$ is a contact manifold.
We denote by $\inf(\alpha)$
the
{\bf minimal period of a closed Reeb orbit}
on $(M,\alpha)$.
Because
$(M,\alpha)$
appears as a hypersurface of restricted contact type
in the present context
$\inf(\alpha)$ is the minimal positive
action of
a closed characteristic
on $M$.

\begin{thm}
  \label{spacing}
  Let $n\geq2$.
  Let
  $(M,\alpha)=(M_1,\alpha_1)\sqcup(M_2,\alpha_2)$
  be a closed hypersurface in $\R^{2n}$
  with two connected components $M_1$ and $M_2$
  which are
  of restricted contact type
  and bound disjoint compact domains
  $
  W=W_1\sqcup W_2
  $
  in $\R^{2n}$, resp.
  If $M$ is contained in the open ball $B_R$
  of radius $R$
  then
  \[
  \inf(\alpha_1)+\inf(\alpha_2)<\pi R^2.
  \]
\end{thm}

\begin{proof}
  The proof is an application
  of the compactness result in \cite{howyze03}.
  We assume the contact form $\alpha$
  to be generic in the sense
  that $1$ is not an eigenvalue
  of the linearized Poincar\'e return map
  for all closed Reeb orbits on $(M,\alpha)$.
  If $\alpha$ is not generic
  we replace $(M,\alpha)$
  by the graph of a positive function on $M$  
  inside a symplectic tubular neighbourhood
  $\big((-\varepsilon,\varepsilon)\times M,\rmd(\rme^s\alpha)\big)$.
  In view of \cite[Proposition 6.1]{howyze98}
  there is a dense set of positive functions on $M$
  such that the contact form
  obtained by restriction of $\rm e^s\alpha$
  to its graphs is generic.
  An application of the Arzel\`a-Ascoli theorem,
  see \cite{hoze94},  
  and the Liouville flow induced by $\alpha$
  allow to undo the perturbation.
  
  For notational convenience we assume $R=1$
  so that $(M,\alpha)$ is a hypersurface
  of restricted contact type
  in the open unit ball $B$.
  Invoking an argument
  used in \cite[Corollary 3.7]{geizeh13}
  we find a primitive $1$-form
  $\lambda$ of $\wst$
  which is equal to $\lamst$
  on a neighbourhood of $\R^{2n}\setminus B$
  such that $\lambda|_{TM}=\alpha$.
  Collapsing the boundary sphere $S^{2n-1}$
  to the hyperplane $\C P^{n-1}$ at infinity
  yields a symplectic embedding $B\subset\C P^n$,
  where $\C P^n$ is provided
  with the Fubini-Study symplectic form $\omega$.
  Recall that
  \[
  \int_{\C P^1}\omega=\pi,
  \]
  that $\C P^n$ is a monotone symplectic manifold,
  and that through any two distinct points
  in $\C P^n$ it passes a unique complex line.
  With \cite[0.2.B]{grom85} we have
  that for any compatible almost complex structure
  on $\C P^n$ and any pair of distinct 
  points $p_1\in W_1$ and $p_2\in W_2$
  there exists a possibly non-unique holomorphic sphere
  through $p_1$ and $p_2$,
  which is homologous to $\C P^1$.
  
  For each $N\in\N$
  we choose an almost complex structure $J_N$
  which is equal to the complex structure of
  $\C P^n$ restricted to $\C P^{n-1}$.
  Moreover, in a neighbourhood of $M$
  the almost complex structure $J_N$
  is subject to the process of
  {\bf stretching the neck}:
  A neighbourhood of $M\subset B$
  is symplectomorphic to
  $\big([-\varepsilon,\varepsilon]\times M,\rmd(\rme^s\alpha)\big)$
  for $\varepsilon>0$.
  We assume that the points $p_1$ and $p_2$
  inside $W$ are contained
  in the complement of this neighbourhood.
  Denote by $V$ the concave filling
  cut out of $\C P^n$ by $(M,\alpha)$.
  We form a symplectic manifold
  \[
  W\cup\big([-N,N]\times M\big)\cup V
  \]
  by identifying the $\varepsilon$-collar
  neighbourhoods of $M$
  in $W$ and $V$ with the
  corresponding $\varepsilon$-collars
  on the {\bf neck}
  \[
  [-N-\varepsilon,N+\varepsilon]\times M,
  \]
  see \cite[p.\ 273--276]{geig08}.
  The symplectic form on $W\cup V$ is $\omega$
  and on the neck $\rmd (\tau\alpha)$,
  where $\tau$ is a smooth
  strictly increasing function
  on $[-N-\varepsilon,N+\varepsilon]$
  that equals $\rme^{s+N}$
  on $[-N-\varepsilon,-N-\varepsilon/2]$ and
  $\rme^{s-N}$ on $[N+\varepsilon/2,N+\varepsilon]$.
  The resulting manifold is symplectomorphic to $\C P^n$.
  A symplectomorphism is obtained
  by following the Liouville flow
  parallel to the $\R$-direction,
  see \cite[p.\ 158]{howyze03}.
  To finish the construction of $J_N$
  it suffices to define $J_N$
  on $[-N-\varepsilon,N+\varepsilon]\times M$.
  Choose a complex structure $j$
  on the contact structure $\ker\alpha$
  that is compatible with $\rmd\alpha$.
  By definition $J_N$ is
  the unique translation
  invariant almost complex structure
  which sends $\partial_s$
  to the Reeb vector field of $\alpha$
  and coincides with $j$ on $\ker\alpha$.
  Under the identifying symplectomorphism
  this defines $J_N$ near $M$.
  On
  $\C P^n\setminus\big((-\varepsilon,\varepsilon)\times M\big)$
  the almost complex structure $J_N$
  does not depend on $N$.
  
  By the above discussion
  we find for each $N\in\N$
  a $J_N$-holomorphic map
  \[
  w_N\co\C P^1\lra\C P^n,
  \qquad
  C_N:=w_N(\C P^1),
  \]
  such that $p_1,p_2\in C_N$,
  $C_N$ intersects $\C P^{n-1}$ positively
  in exactly one point,
  and $C_N$ has {\bf energy}
  \[
  \int_{C_N}\omega=\pi.
  \]
  We quote the compactness result
  on \cite[p.\ 192--193]{howyze03},
  which applies to the present situation
  because $M$ is of restricted contact type
  in $B$.
  Therefore,
  formulated in the language of \cite{behwz03},
  a subsequence of $w_N$
  converges to a holomorphic building.
  The {\bf lowest level} of the building,
  which corresponds to
  components in
  $W\cup\big([0,\infty)\times M\big)$,
  consists of finite energy planes only,
  cf.\ \cite[p.\ 193, Fig.\ 14]{howyze03}.
  At least two of them, say $u_1$ and $u_2$,
  pass through the points $p_1$ and $p_2$,
  resp. 
  Moreover,
  the total {\bf Hofer-energy}
  \[
  E(u)=\sup_{\tau}\int_{\C}u^*\omega_{\tau}
  \]
  satisfies
  \[
  E(u_1)+E(u_2)<\pi,
  \]
  see \cite{howyze03}.
  Here the supremum is taken
  over all smooth strictly
  increasing functions $\tau$
  on $[-\varepsilon,\infty)$
  that agree with $\rme^s$
  on $[-\varepsilon,-\varepsilon/2]$
  and tend to $1$ as $s\ra\infty$.
  The symplectic form $\omega_{\tau}$
  equals $\rmd\lambda$
  on $W\setminus\big([-\varepsilon,0]\times M\big)$
  and $\rmd(\tau\alpha)$
  on $[-\varepsilon,\infty)\times M$.
  Because the primitive $\lambda$
  extends to $\tau\alpha$ on the cylindrical end
  the finite energy planes $u_1$ and $u_2$
  are asymptotic to closed Reeb orbits
  of period less or equal to its Hofer-energy,
  see \cite{hof93} and cf.\ \cite[Lemma 6.3]{geizeh13}.
  In other words,
  we have found closed Reeb orbits
  in each component of
  $(M_1,\alpha_1)\sqcup (M_2,\alpha_2)$
  having period $T_1$ and $T_2$, resp.,
  such that $T_1+T_2<\pi$.
\end{proof}


\subsection{Proof  of Theorem \ref{spherespacing}\label{applandthm11}}

The image $S_r$ of a symplectic embedding of $S^{2n-1}_r$
is a hypersurface of restricted contact type.
This follows with the Mayer-Vietoris sequence
for the de Rham cohomology,
cf.\ the proof of \cite[Corollary 3.7]{geizeh13}.
Moreover, the minimal positive action equals $\pi r^2$.
Theorem \ref{spacing}
implies that the sum of the smallest actions
of $S_{r_1}$ and $S_{r_2}$ is bounded
by the action of $S_R^{2n-1}$.
Therefore, Theorem \ref{spherespacing} follows.


\subsection{Superadditivity\label{superadd}}

Smooth boundaries of bounded convex domains $K$
in $\R^{2n}$
are of restricted contact type.
Moreover,
the action-capacity representation theorem
for the Hofer-Zehnder capacity $c_{\hz}$
implies that $c_{\hz}(K)$
is the minimal positive action of a
closed characteristic on $\partial K$,
see \cite{hoze90}.
Because $c_{\hz}$ has inner regularity
Theorem \ref{spacing} yields that
\[
c_{\hz}(K_1)+c_{\hz}(K_2)\leq c_{\hz}(B_R)
\]
for disjoint convex subsets
$K_1$ and $K_2$ of $B_R$.
Artstein-Avidan and Ostrover \cite{artost08}
proved that
\[
c_{\hz}(K_1)^{1/2}+c_{\hz}(K_2)^{1/2}\;\leq c_{\hz}(K_1+K_2)^{1/2}
\]
for bounded convex sets
without assuming that
$K_1$ and $K_2$ are disjoint.
Furthermore,
\cite[Corollary 1.3]{jiang99}
says that
\[
c_{\hz}(U_1)+c_{\hz}(U_2)\leq c_{\hz}(B_R)
\]
for open disjoint subsets
$U_1$ and $U_2$ of $B_R$.
In particular,
$c_{\hz}$ tends to zero on
$B_1\setminus\overline{B_{1-\varepsilon}}$
as $\varepsilon\ra 0$
while the spherical capacity
\cite{zehzil12}
as well as the orbit capacity
\cite{geizeh13,geizeh12}
are equal to $\pi$
for all $\varepsilon\in(0,1)$.
Hence,
neither the spherical capacity
nor the orbit capacity
equal the Hofer-Zehnder capacity.


\subsection{Codisc radii\label{pwcodb}}

In the proof of Theorem \ref{spacing}
the existence of finite energy surfaces
contained in the symplectic filling $W$
follows without making use
of the restricted contact type
property of $M\subset B$.
In order to obtain
the estimates on the periods
it suffices that the filling $W$
is exact.
Hence,
Theorem \ref{spacing}
continues to hold e.g.\
for images of codisc bundles.

\begin{cor}
  \label{behwzcase}
  Let $n\geq2$.
  Let $M=M_1\sqcup M_2$
  be a closed hypersurface in $\R^{2n}$
  with two connected components
  that bound disjoint compact domains $W=W_1\sqcup W_2$.
  Let $\lambda$ be a primitive $1$-form of $\wst$
  on $W$
  such that $\alpha_1=\lambda|_{TM_1}$
  and $\alpha_2=\lambda|_{TM_2}$
  are contact forms.
  If $M\subset B_R$ then
  \[
  \inf(\alpha_1)+\inf(\alpha_2)<\pi R^2.
  \]
\end{cor}

\begin{proof}
  We assume the splitting situation
  of the proof of Theorem \ref{spacing}
  and consider again the sequence $w_N$
  of $J_N$-holomorphic spheres.
  The energy
  \[
  \int_{C_N}\omega=\int w_N^*\omega+\int w_N^*\rmd(\tau\alpha)
  \]
  decomposes into the sum of two integrals.
  The first one is taken over
  \[
  w_N^{-1}\big(\C P^n\setminus([-\varepsilon,\varepsilon]\times M)\big)
  \]
  and the second over
  \[
  w_N^{-1}\big([-N-\varepsilon,N+\varepsilon]\times M)\big).
  \]
  Here $\tau$ is a smooth increasing
  function on $[-N-\varepsilon,N+\varepsilon]$
  that equals $\rme^{s+N}$
  on $[-N-\varepsilon,-N-\varepsilon/2]$ and
  $\rme^{s-N}$ on $[N+\varepsilon/2,N+\varepsilon]$.
  Notice that the integral is independent
  of the choice of $\tau$,
  cf.\ \cite[p.\ 159]{howyze03}.
  Smoothing the function
  that is defined to be
  $\rme^{s+N}$ on $[-N-\varepsilon,-N]$,
  $1$ on $[-N,N]$, and
  $\rme^{s-N}$ on $[N,N+\varepsilon]$,
  by functions $\tau$ as described
  we get
  \[
  \int_{C_N}\omega=
  \int w_N^*\omega+\int w_N^*\rmd\alpha,
  \]
  where the first integral is taken over
  \[
  w_N^{-1}(\C P^n\setminus M)
  \]
  and the second over
  \[
  w_N^{-1}\big([-N,N]\times M\big).
  \]
  By \cite[Lemma 9.2 and Theorem 10.3]{behwz03}
  there exists a subsequence of $w_N$
  which converges to a holomorphic building.
  Its lowest level contains energy surfaces $u_1$ and $u_2$
  such that $p_1\in \im(u_1)$ and $p_2\in \im(u_2)$.
  \cite[Proposition 5.6]{behwz03}
  implies that these energy surfaces
  are asymptotic to finitely many
  periodic Reeb orbits on $(M,\alpha)$
  with total period $T$.
  With \cite[Lemma 9.1]{behwz03}
  their $\omega$-{\bf energy}
  \[
  E_{\omega}(u_1)+E_{\omega}(u_2)
  \]
  is less than $\pi$,
  where
  \[
  E_{\omega}(u)=
  \int_{u^{-1}(W)} u^*\rmd\lambda
  +\int_{u^{-1}\big([0,\infty)\times M\big)}u^*\rmd\alpha.
  \]
  Therefore,
  by smoothing out the integrand
  and employing an approximation argument as above,
  we see that the total
  period $T$ is less than $\pi$.
\end{proof}

\begin{rem}
  \label{lagrscelet}
  Assuming the situation of
  Corollary \ref{behwzcase}
  let $Y$ be the Liouville vector
  field on $(W,\wst)$
  defined by $\lambda$.
  Because $W$ is compact
  the flow of $Y$ exists on $(-\infty,0]$.
  Therefore,
  $W$ decomposes
  into the Lagrangian skeleton $Y^{-1}(0)$
  and the negative
  half-symplectization
  $\big((-\infty,0]\times M,\rmd(\rme^s\alpha)\big)$.
  If $Y^{-1}(0)$ represents
  a cycle of dimension at most $2n-3$
  then the finite
  energy surfaces $u_1$ and $u_2$
  obtained in the proof
  of Corollary \ref{behwzcase}
  can be homotoped with
  fixed (asymptotic) boundary conditions into $M$.
  This can be used
  to estimate the {\bf minimal total period}
  ${\inf}_{\ell}(\alpha)$
  {\bf of a null-homologous Reeb link}
  in $(M,\alpha)$
  introduced in \cite{geizeh12}.
  Theorem \ref{spacing} and
  Corollary \ref{behwzcase} generalize accordingly.
\end{rem}

\begin{rem}
  \label{reltogoe}
  We consider a
  Riemannian manifold $(L,g)$.
  Using the metric $g$
  we identify the
  tangent bundle of $L$
  with $T^*L$.
  The canonical Liouville $1$-form of $T^*L$
  induces a contact form
  \[
  \alpha=\lambda_{\can}|_{TS_r^*(g)L}
  \]
  on the cosphere bundle $S_r^*(g)L$
  of radius $r$.
  According to \cite{geig08}
  non-trivial closed geodesics
  on $(L,g)$ and
  closed Reeb orbits
  on $S_r^*(g)L$
  are in one-to-one correspondence.
  The {\bf speed curve} $\bar{c}$
  of a closed geodesic $c$,
  which is parametrized
  proportional to arc length,
  with speed $|\dot c|=r$
  is containd in $S_r^*(g)L$.
  It defines a closed Reeb orbit
  $\gamma$ by reprarametrizing
  $\bar{c}$ by $1/r^2$.
  The action of $\gamma$
  and the length of $c$
  are related via
  \[
  \int_{\gamma}\alpha=
  \int_{\bar{c}}\lambda_{\can}=
  r\length(c).
  \]
  The
  {\bf length of the shortest
    non-trivial closed geodesic}
  on $(L,g)$ is denoted by $\inf(g)$,
  which is bounded by the injectivity radius
  from below.
  We have
  \[
  r\inf(g)=\inf(\alpha).
  \]
  
  Assume in the following
  that $L$ decomposes into
  closed submanifolds
  $L_1\sqcup L_2$.
  The metric $g$
  defines Riemannian manifolds
  $(L_1,g_1)$ and $(L_2,g_2)$.
  If
  the closure of the codisc bundles
  \[
  D_{r_1}^*(g_1)L_1\sqcup D_{r_2}^*(g_2)L_2
  \]
  embed symplectically
  into $B_R$ such that the images are disjoint
  Corollary \ref{behwzcase} implies
  \[
  r_1\inf(g_1)+r_2\inf(g_2)<\pi R^2.
  \]
  If in addition $(L,g)$
  has no contractible closed geodesics
  we obtain in view of Remark \ref{lagrscelet}
  that $2r\inf(g)\leq {\inf}_{\ell}(\alpha)$
  using the bundle projection.
  This implies
  \[
  2r_1\inf(g_1)+2r_2\inf(g_2)<\pi R^2.
  \]
  In Section \ref{cora} we continue
  the discussion on the size of
  symplectically embedded codisc bundles.
\end{rem}


\subsection{More than two components\label{more}}

We consider a closed hypersurface
\[
M=M_1\sqcup\ldots\sqcup M_k
\]
of $B\subset\R^{2n}$
with $k$ connected components.
We assume that the bounded components
$W_1,\ldots,W_k$
of the complements of
$M_1,\ldots,M_k$
are pairwise disjoint
and that $\wst$
admits a primitive $1$-form
on the closure of
$W=W_1\sqcup\ldots\sqcup W_k$
that restricts to contact forms
$\alpha_1,\ldots,\alpha_k$
on 
$M_1,\ldots,M_k$,
resp.
In Corollary \ref{behwzcase} we
considered the case $k=2$.
The proof of
Theorem \ref{spacing}
and Corollary \ref{behwzcase}
generalizes to hypersurfaces $M$
with $k\geq 3$ connected components
provided that there exists
a holomorphic curve $C$
through $k$ generic points
for any (generic) compatible
almost complex structure.
Therefore,
\[
\inf(\alpha_1)+\ldots+\inf(\alpha_k)
<\int_C\omega.
\]
Observe that the compactness result
in \cite{behwz03}
which we used in the
above proofs
applies to
holomorphic curves of higher genus.

\begin{defn}
  The smallest positive
  action of a closed characteristic
  on $M$ divided by $\pi$
  is denoted by $a_k$.  
\end{defn}

It follows form Corollary \ref{behwzcase}
that $a_k<1/2$ for all $k$.

\begin{ex}
  We consider $B\subset\C P^2$.
  Through $k=3d-1$ generic points
  there exists a
  holomorphic sphere
  of degree $d$,
  which has area
  $\int_C\omega=d\pi$,
  see \cite[Proposition 7.4.8]{mcsa04}.
  Therefore,
  \[
  a_{3d-1}<\frac{d}{3d-1}.
  \]
  Taking genus $\tfrac12(d-1)(d-2)$ curves
  which pass through $k=\frac12d(d+3)$
  points in general position
  and
  whose symplectic area equals $d\pi$
  into account we get
  \[
  a_{\tfrac{d(d+3)}{2}}<\frac{2}{d+3},
  \]
  see \cite[0.2.B]{grom85} and \cite{hls97}.
\end{ex}


\section{The size of a Weinstein neighbourhood}
\label{tsofweinngbh}

The aim of this section
is to prove Theorem \ref{isopermest}
and Theorem \ref{codiscradcap}.


\subsection{The action-area inequality\label{aaineq}}

Let $(X,\omega)$ be a symplectic manifold
which is symplectically aspherical,
i.e.\ the {\bf symplectic area}
$\int_{S^2}f^*\omega$
vanishes for all smooth maps $f\co S^2\ra X$.
We assume that $(X,\omega)$ is either closed,
of bounded geometry
in the sense of Gromov \cite{grom85},
or compact with convex
contact type boundary.
In the latter case we replace
$X$ by its completion
so that $(X,\omega)$
has positive cylindrical ends
as introduced in \cite{behwz03}.

Let $L\subset X$
be a closed Lagrangian submanifold.
The {\bf Gromov width of}
$L$ is defined by
\[
\sigma(L)=\sup_J\sigma(L,J),
\]
where $\sigma(L,J)$ is the {\bf minimal symplectic area}
$\int_Du^*\omega$
of a non-constant $J$-holomorphic
disc $u\co D\ra X$
with boundary on $L$,
see \cite{grom85}.
The supremum
is taken over all
almost complex structures $J$
that are tamed by $\omega$
and have adapted
boundary or asymptotic conditions,
resp.
Notice that $\sigma(L,J)=\infty$
if no such disc exists
and that $\sigma(L,J)>0$
by Gromov's compactness theorem,
cf.\ \cite{fra08}.

Denote by $D^*L$
the unit codisc bundle of $L$
w.r.t.\ a Riemannian metric.
On the unit cotangent bundle $S^*L$
the canonical Liouville $1$-form
$\lambda_{\can}$ defines
a contact form
\[
\alpha=\lambda_{\can}|_{TS^*L}.
\]
The aim is to compare
the Gromov width $\sigma(L)$
with the
minimal period of a closed Reeb orbit
$\inf(\alpha)$.

\begin{thm}
  \label{aaineqthm}
  If the closure of $D^*L$
  embeds symplectically
  into $(X,\omega)$
  then
  \[
  \inf(\alpha)<\sigma(L).
  \]
\end{thm}

\begin{proof}
  The proof is based on
  a stretching the neck argument
  along the lines of Theorem \ref{spacing}.
  Denote by $M$ the image of $S^*L$
  and assume that
  the contact form $\alpha$
  on $M$ is generic.
  Identify $\overline{D^*L}$
  with its image $W$ in $X$ and
  denote the Liouville primitive
  of $\omega|_W$ by $\lambda$.
  Set $V=\overline{X\setminus W}$
  so that $X$ decomposes as $W\cup_MV$.
  
  For each $N\in\N$
  we define a compatible almost complex structure
  $J_N$ on $(X,\omega)$
  as in the proof of Theorem \ref{spacing}
  such that the sequence $J_N$
  only depends on $N$
  in the distinguished neigbhourhood of $M$.
  We choose $J_N$ to be cylindrical, resp.,
  to ensure uniform
  $C^0$-bounds on all holomorphic discs.
  This requires a modification of $J_N$
  in a neighbourhood of $\partial V\setminus M$,
  resp.,
  near the ends of $V$.
  
  We assume that
  $\sigma(L)$ is finite.
  Hence,
  there is a sequence
  of $J_N$-holomorphic discs
  \[
  w_N\co (D,\partial D)\lra (X,L)
  \]
  with energy
  \[
  0<\int_Dw_N^*\omega\leq\sigma(L).
  \]
  Moreover, we assume that $w_N(0)$ is contained in
  $V\setminus\big([0,\varepsilon]\times M\big)$.
  
  As in \cite[p.\ 163]{howyze03}
  we choose a Riemannian metric
  on $X$ of bounded geometry
  which is independent of $N$ on
  $X\setminus\big([-\varepsilon,\varepsilon]\times M\big)$
  and is equal to
  a product metric
  on the neck $[-N,N]\times M$.
  An application
  of the mean value theorem
  to the path $w_N(t)$,
  $t\in[0,1]$,
  shows
  that there are no
  uniform gradient bounds on $w_N$.
  In other words,
  after passing to a subsequence $w_{\nu}$,
  there exists
  a sequence $z_{\nu}\ra z_0$ in $\bar{D}$
  such that
  \[
  R_{\nu}=|\nabla w_{\nu}(z_{\nu})|\lra\infty.
  \]
  We call $z_0$ a {\bf bubbling off point}.
  
  We claim that there are
  only finitely many bubbling off points.
  In view of \cite[Lemma 3.2]{howyze03}
  it is enough to show that
  there exists $c>0$
  such that for any
  (subsequence of a)
  bubbling off sequence $z_{\nu}\ra z_0$
  and for any $\varrho>0$
  \[
  \liminf_{\nu\ra\infty}
  \int_{D_{\varrho}(z_{\nu})}w_{\nu}^*\omega>c.
  \]
  If a bubbling off point
  is contained in the interior of $D$
  the bubbling off argument
  on \cite[p.\ 163--167]{howyze03}
  shows that there exists
  a finite energy plane $v$
  with Hofer-energy $E(v)\leq\sigma(L)$
  in $W\cup\big([0,\infty)\times M\big)$,
  $\R\times M$,
  or $\big((-\infty,0]\times M\big)\cup V$.
  In the first two cases
  we get $\inf(\alpha)\leq E(v)$;
  in the third,
  invoking the compactness theorem
  \cite[Theorem 10.5]{behwz03},
  $E(v)$ is bounded from below
  by a uniform positive constant.
  If a bubbling off point
  is contained on the boundary
  $\partial D$ we distinguish
  following \cite{hof93,geizeh10}
  two cases:
  We view $w_{\nu}$
  as a $J_{\nu}$-holomorphic map
  on the upper half plane $H^+$
  such that the bubbling off point equals $0$.
  Using Hofer's Lemma
  \cite[Lemma 6.4.5]{hoze94}
  we modify $z_{\nu}=x_{\nu}+\rmi y_{\nu}$
  such that
  \[
  R_{\nu}y_{\nu}\lra r
  \]
  for some $r\in[0,\infty]$,
  and that there exists
  a sequence $\varepsilon_{\nu}\searrow 0$
  with $\varepsilon_{\nu}R_{\nu}\ra\infty$
  and
  \[
  |\nabla w_{\nu}(z)|\leq 2R_{\nu}
  \]
  for all $z\in H^+$
  with $|z-z_{\nu}|\leq\varepsilon_{\nu}$.
  
  The first case is $r=\infty$.
  With the rescaling argument
  on \cite[p.\ 560]{geizeh10}
  we obtain a finite energy plane $v$ in
  $W\cup\big([0,\infty)\times M\big)$,
  $\R\times M$,
  or $\big((-\infty,0]\times M\big)\cup V$,
  which has Hofer-energy $E(v)$
  uniformly bounded from below
  as in the above argument.
  It remains to consider
  the case $r<\infty$.
  Replace the sequence $w_{\nu}$
  by the rescaled sequence
  \[
  u_{\nu}(z)=w_{\nu}\big(x_{\nu}+z/R_{\nu}\big).
  \]
  Set $\zeta_{\nu}=\rmi R_{\nu}y_{\nu}$,
  and observe
  that $\zeta_{\nu}\ra\rmi r$
  and $|\nabla u_{\nu}(\zeta_{\nu})|=1$.
  Hence we get
  \[
  |\nabla u_{\nu}(z)|\leq2
  \]
  for all $z\in H^+$
  with $|z-\zeta_{\nu}|\leq\varepsilon_{\nu}R_{\nu}$.
  Identifying $[-N,N]\times M$
  with $[0,2N]\times M$
  we see
  \[
  2\nu\leq\dist\big(u_{\nu}(0),\{2\nu\}\times M\big).
  \]
  With the mean value theorem this implies
  \[
  \nu\leq\dist\Big(0,u_{\nu}^{-1}\big(\{2\nu\}\times M\big)
  \cap D_{\varepsilon_{\nu}R_{\nu}}(\zeta_{\nu})\Big)
  \]
  for all sufficiently large $\nu$.
  In other words,
  each $u_{\nu}$ maps the half-disc $D_R^+$
  into $W\cup\big([0,\infty)\times M\big)$
  provided $R\ll\nu$.
  Hence,
  a subsequence of $u_{\nu}$
  converges in $C^{\infty}_{\loc}$
  to a non-constant holomorphic map
  \[
  u\co (H^+,\R)\lra
  \Big(W\cup \big([0,\infty)\times M\big),L\Big),
  \]
  see \cite[Proposition 6.1 Case 1.1]{geizeh10}
  and \cite[p.\ 168]{howyze03}.
  With the mean value inequality,
  see \cite[Remark 3.54]{abbasbook},
  and the argument before
  \cite[Lemma 6.2]{geizeh10}
  a neighbourhood of $\infty\in H^+$
  is mapped by $u$ into a compact neighbourhood of $L$.
  In view of the finiteness of the Hofer energy of $u$
  the boundary removable of singularities theorem
  \cite[Theorem 4.1.2]{mcsa04} applies.
  That means $u$ extends to a non-constant
  holomorphic disc map with boundary on $L$.
  Because $W\cup\big([0,\infty)\times M\big)$
  provided with $\omega_{\tau}$
  is symplectomorphic to $D^*L$,
  see \cite[Lemma 2.10]{howyze03},
  this s a contradiction,
  i.e.\ the case $r<\infty$ can not occur.
  Consequently,
  there are only finitely many bubbling off points.
  
  Denote the finite set
  of bubbling off points
  by $\Gamma\subset\bar{D}$.
  Recall that $\Gamma\neq\emptyset$.
  In the complement
  of any neighbourhood
  of $\Gamma$ the sequence $w_{\nu}$
  admits uniform gradient bounds.
  Applying the mean value theorem
  we get $C^0$-bounds
  such that a subsequence $w_{\nu}$
  converges in
  $C^{\infty}_{\loc}(\bar{D}\setminus\Gamma)$
  to a punctured holomorphic
  disc $w$ in
  $W\cup\big([0,\infty)\times M\big)$
  with boundary in $L$.
  The Hofer-energy $E(w)$
  is strictly bounded from
  above by $\sigma(L)$.
  We claim that $w$
  is not constant.
  Observe that
  for $\varrho>0$
  sufficiently small
  and $z\in\Gamma$
  we have a uniform bound
  \[
  \liminf_{\nu\ra\infty}\int_{D_{\varrho}(z)}w_{\nu}^*\omega>c.
  \]
  Arguing by contradiction
  we see that all the circles,
  resp., chords
  $w_{\nu}(\partial D_{\varrho}(z))$
  converge in $C^{\infty}$
  to a point in $L$.
  In both cases as
  on \cite[p.\ 85--86]{mcsa04}
  we can extent $w_{\nu}(D_{\varrho}(z))$
  smoothly to sphere maps into $X$.
  If $\nu\gg 1$
  we can assume
  that the symplectic areas
  are positive.
  This contradicts our assumption
  that $(X,\omega)$
  is symplectically aspherical.
  Therefore,
  $w$ is a non-constant
  punctured holomorphic disc.
  All its boundary singular points
  can be removed by the above argument.
  We assume that all its
  removable interior punctures
  are removed as well.
  With \cite[Proposition 2.11]{howyze03}
  $w$ is a finite energy disc
  in $W\cup\big([0,\infty)\times M\big)$
  with boundary on $L$ and positive punctures.
  Taking the primitive $\lambda$ into account,
  which vanishes along $L$,
  an application
  of Stokes's theorem yields
  \[
  \inf(\alpha)\leq E(w)<\sigma(L).
  \]
  This proves the Theorem \ref{aaineqthm}.
\end{proof}


\subsection{Proof  of Theorem \ref{isopermest}\label{appltothm12}}

In the two dimensional case
a closed connected Lagrangian
submanifold $L$
is an embedded curve in the plane.
The isoperimetric inequality
implies that the
enclosed bounded domain $D$
has area
$\leq\frac{1}{4\pi}\length(L)^2$.
Notice,
that $L$ divides $D_r^*L$
into two components of equal area.
Precisely one component
is mapped into $D$.
Since a symplectomorphism
preserves the area
the area of $D_r^*L$
is $\leq\frac{1}{2\pi}\length(L)^2$.
Because the metric on $L$
is a positive multiple of the
metric induced by $\R/2\pi\Z$
there exists $\varepsilon>0$
such that
the area of $D_r^*L$
equals $4\pi\varepsilon r$
and $\inf(g)=2\pi\varepsilon$.
It follows that
\[
r\leq\frac{1}{4\pi}
\frac{\length(L)^2}{\inf(g)}.
\]

Let $2n\geq4$.
Consider a symplectic embedding
of the closure of $D_r^*L$
into $\R^{2n}$.
Notice,
that the Lagrangian
submanifold $L$
is displaceable.
A theorem of
Chekanov \cite{chek98}
implies that the
Gromov width $\sigma(L)$
is bounded by
the displacement energy $d(L)$ of $L$.
In \cite{vit00} Viterbo proved
an isoperimetric inequality
$d(L)^n\leq\rho_n\vol(L)^2$
for a
positive constant $\rho_n$.
As explained in
Remark \ref{reltogoe}
we have
$r\inf(g)=\inf(\alpha)$
for the contact form
$\alpha=\lambda_{\can}|_{TS_r^*L}$.
Theorem \ref{aaineqthm} yields
\[
\big(r\inf(g)\big)^n\leq\rho_n\vol(L)^2.
\]
This proves Theorem \ref{isopermest}.

\begin{rem}
  \label{vitest}
  In \cite{vit00} Viterbo proved
  for the volume $\vol(L)$
  of a closed Lagrangian submanifold
  $L$ in $\R^{2n}$
  w.r.t.\ the induced metric $g$
  that
  \[
  d(L)^n\leq\rho_n\vol(L)^2
  \]
  with a positive constant
  \[
  \rho_n\leq\sqrt{2^{n(n-3)}}\; n^n.
  \]
  With the inequality $r\inf(g)\leq d(L)$ 
  we get for the radius
  of a symplectically embedded
  codisc bundle taken
  w.r.t.\ the induced metric
  \[
  \frac{\vol(L)^2}{\inf(g)^n}
  \geq\frac{r^n}{\rho_n}.
  \]
  This inequality remains
  valid for all Riemannian metrics $g$
  induced by any Hamiltonian
  deformation of $L$.
  As \'Alvarez Paiva explained
  to the author
  a computation of the greatest
  value of $r$ in the above
  inequality is related to
  questions in \emph{systolic geometry}.
\end{rem}


\subsection{Non-embeddability of the cotangent bundles}
\label{noembofcobundles}

Let $(X,\omega)$
be a symplectically aspherical
symplectic manifold
as described in Section \ref{aaineq}.
In \cite{chek98}
Chekanov proved
for displaceable
Lagrangian submanifolds $L$
the inequality
\[
0<\sigma(L)\leq d(L)<\infty
\]
for the displacement
energy $d(L)$ of $L$.

\begin{cor}
  \label{aaineqcor}
  Let $L\subset(X,\omega)$
  be a closed displaceable Lagrangian submanifold.
  Then there is no symplectic embedding
  of $T^*L$ into $(X,\omega)$ relative $L$.
\end{cor}

\begin{proof}
  Let $g$ be a metric on $L$.
  Arguing by contradiction
  we find for all positive $r$
  a symplectic embedding of the
  $r$-codisc bundle of $L$.
  With Theorem \ref{aaineqthm}
  and Chekanov's result \cite{chek98}
  we find
  \[
  r\inf(g)=
  \inf\big(\lambda_{\can}|_{TS^*_r(g)L}\big)
  \leq d(L).
  \]
  Letting $r$ tend to
  infinity yields a contradiction.
\end{proof}

\begin{rem}
  In the particular case
  the symplectic form
  $\omega=\rmd\lambda$
  on $X$ is exact
  the restriction of $\lambda$
  to $TL$ is a closed $1$-form on $L$.
  Its cohomology class $\lambda_L$,
  the so-called {\bf Liouville class},
  cf.\ \cite{polt01},
  is independent of the
  choice of the primitive $\lambda$
  provided $X$ is simply connected.
  Recall,
  that a norm on the space
  of cohomology $1$-classes
  $m$ can be defined
  via
  $\|m\|=\inf\{\sup_L|\mu|\,|\,\mu\in m\}$,
  cf.\ \cite{bates98}.
  If a neighbourhood of the closure
  of the $\|\lambda_L\|$-codisc bundle of $L$
  embeds symplectically relative $L$
  the image $L_{\lambda}$
  of the section into $T^*L$ representing $-\lambda_L$
  is an exact Lagrangian
  submanifold of $(X,\rmd\lambda)$,
  see \cite[Section 7]{arn86}.
  This was pointed out to the author by Polterovich.
  With Chekanov's result \cite{chek98}
  $L_{\lambda}$ is not displaceable.
  In particular,
  no subcritical Stein manifold contains a
  symplectically embedded cotangent
  bundle of a closed manifold,
  cf.\ \cite{cieleli12}.
  Corollary \ref{aaineqcor} serves as a
  generalization to the
  symplectically aspherical case.
\end{rem}

\begin{rem} 
  To give an example of non-embeddability
  of the cotangent bundles
  in the presence of holomorphic spheres
  we make the following remark.
  Barraud, Biran and Cornea
  \cite{barcor06,barcor07,bircor09,bircor09b}
  defined the {\bf relative Gromov radius} $c_B(L)$
  of a closed Lagrangian submanifold $L$
  in a symplectic manifold $(V,\omega)$
  to be the supremum over all $\pi r^2$
  such that there exists a
  symplectic embedding
  $\varphi\co B_r\ra V$
  with $\varphi^{-1}(L)$
  equal to $B_r\cap\R^n\subset\C^n$.
  The relative Gromov radius
  of the zero section of $T^*L$
  is not finite.
  Hence,
  if a cotangent bundle embeds symplectically
  into $(V,\omega)$
  the relative Gromov radius
  of the image of the zero section
  must be infinite in $(V,\omega)$.  
  Consider a closed Lagrangian submanifold
  $L$ in $T^*Q\times\C P^1$,
  where $Q$ is any closed manifold.
  Then $c_B(L)$ is bounded by
  the absolute Gromov radius
  $c_B(T^*Q\times\C P^1)=\pi$.
  Hence,
  no cotangent bundle does embed
  symplectically into $T^*Q\times\C P^1$.
  
  Notice, that it is not known in general
  whether the relative Gromov radius
  of a closed Lagrangian submanifold
  $L$ in $\R^{2n}$ is finite.
  Examples in the monotone case
  can be found in \cite{bircor09b,buh10,damian12}.
  Its spherical variant will be discussed
  in Section \ref{relsphgromrad}.
\end{rem}


\subsection{Proof of Theorem \ref{codiscradcap}\label{cora}}

We consider the torus
$\R^n/2\pi\Z^n$
with the metric
induced from $\R^n$
so that
the shortest
non-trivial closed geodesics
have length $2\pi$.
The $n$-fold product of the maps
\[
(q,p)\mapsto\sqrt{1+2p}\;\rme^{\rmi q},
\qquad\text{resp.,}\qquad
\sqrt{1/n+2p}\;\rme^{\rmi q}
\]
embed the
$1/2$-codisc,
resp.,
the $1/2n$-codisc bundle
symplectically into $\R^{2n}$.
The images of the
zero section are
the Clifford tori
$T_1$ and $T_{1/\sqrt{n}}$
which equal the product
of $n$ circles in
$\R^2\times\cdots\times\R^2$
of radius $1$ and $1/\sqrt{n}$,
resp.
The action
on the corresponding
cosphere bundles
induced by the shortest
non-trivial closed geodesics
equal $\pi$ and $\pi/n$,
resp.
This shows that
$c_D(P)\geq\pi$ and
$c_D(B)\geq\frac{\pi}{n}$.

Consider a metric $g$ on $S^1$.
After a reparametrization
there exist a positive
constant $\varepsilon$
such that $g=\varepsilon^2g_0$,
where $g_0$
is the metric induced
by $\R/2\pi\Z$.
Therefore,
$\inf(g)=2\pi\varepsilon$.
Assume that
$D_r^*(g)S^1$
embeds into $\R^2$
preserving the area
such that $S^1$
is mapped into
the open unit disc.
This implies that
$1/2$ times the area of
$D_r^*(g)S^1$
is $<\pi$.
Because the area of
$D_r^*(g)S^1$
is equal to
$4\pi\varepsilon r$
we get $\varepsilon r<1/2$.
Hence,
$r\inf(g)<\pi$.
This proves that
$c_D$ is a capacity in dimension $2$.

If $2n\geq 4$ we argue as follows.
Non-trivial closed
geodesics $c$
on Riemannian manifolds
$(L,g)$
are in one-to-one
correspondence
with closed Reeb orbits
$\gamma$ on
$S_r^*(g)L$
for all positive $r$,
see \cite{geig08}.
The correspondence assigns
to a geodesic $c$
which is assumed to have constant speed
$|\dot c|=r$
the reparametrized
speed curve $\gamma=c\circ 1/r^2$.
The action
$\int_{\gamma}\alpha$,
where
$\alpha=\lambda_{\can}|_{TS_r^*(g)L}$,
equals
$\int_{\bar{c}}\lambda_{\can}=r\length(c)$.
Therefore,
\[
r\inf(g)=\inf(\alpha).
\]
On the other hand
the action-area inequality
in Theorem \ref{aaineqthm}
and Chekanov's result \cite{chek98}
yield
\[
\inf(\alpha)<\sigma(L)\leq d(L).
\]
The displacement energy $d$,
which is known to be a special capacity,
see \cite{hoze94},
takes the value $\pi$
on the open unit ball
and the open unit symplectic cylinder.
Hence,
$c_D(Z)\leq\pi$.
This proves Theorem \ref{codiscradcap}.

\begin{rem}
  Notice that $c_D(L)\leq\sigma(L)$
  and that for subsets $U$ of $\R^{2n}$
  \[
  c_D(U)\leq
  \sigma(U)
  :=\sup\big\{\sigma(L)\,\big|\,
  \text{$L\subset U$}\big\}.  
  \]
  With Hermann's work \cite{herm04}
  one obtains upper bounds for the codisc radius
  in terms of the Floer-Hofer
  resp.\ the Viterbo capacity,
  as well as with \cite{geizeh13,geizeh12}
  in terms of the orbit capacity in dimensions $\geq4$.
\end{rem}


\subsection{Monotone Lagrangian submanifolds}
\label{monolagsub}

By Damian's proof
of the Audin conjecture
in the monotone case \cite{damian12}
the minimal Maslov number
of an orientable monotone Lagrangian
submanifold $L$
which admits a metric
of non-positive sectional curvature
equals $2$.
The minimizing class
can be represented by
a holomorphic disc
with boundary on $L$
for any compatible almost
complex structure.
Therefore,
the {\bf minimal symplectic area}
$\inf(L)$
and the Gromov width $\sigma(L)$ are equal.
Hence,
with \cite[Corollary 2.6]{zeh12}
we get
\[
\sigma(L)=\inf(L)\leq\frac{\pi}{n}
\]
provided $L\subset B$.
With the arguments from
Section \ref{cora}
this yields a
special capacity
$c_{D,\mathrm m}$ such that
\[
c_{D,\mathrm m}(Z)=\pi
\qquad\text{and}\qquad
c_{D,\mathrm m}(B)=\frac{\pi}{n}
\]
by restricting
the definition of $c_D$
to Lagrangian submanifolds
that are orientable,
monotone,
and admit a metric
of non-positive sectional curvature.
We remark
that still
the supremum is taken
over \emph{all} Riemannian metrics
on $L$.
Similar to
\cite[Corollary 3.3]{zeh12}
an application of
the capacity $c_{D,\mathrm m}$
implies:

\begin{cor}
  Let $L\subset B_R$
  be an orientable closed monotone
  Lagrangian submanifold
  which admits a metric
  of non-positive sectional curvature.
  Then for all Riemannian
  metrics $g$ on $L$
  and all radii $r$
  such that the corresponding
  $r$-codisc bundle embeds
  into $\R^{2n}$
  symplectically
  we have
  \[
  r\inf(g)\leq\frac{\pi}{n}R^2.
  \]
\end{cor}


\subsection{Relation to the link capacity\label{link}}

In \cite{zeh12} the link capacity is defined.
A variant of it can be obtained as follows.
For open subsets $U\subset\R^{2n}$
and closed oriented
monotone Lagrangian
submanifolds $L\subset\R^{2n}$
which admit a Riemannian metric $g$
of non-positive sectional curvature
consider symplectic embeddings
of $D_r^*(g)L$ into $U$ relative $L$.
Then
\[
\ell_{\mathrm m}^+(U)
:=\sup\big\{2r\inf(g)\,|\,
\text{$D_r^*(g)L\hookrightarrow U$, $g$, and $r>0$}\big\},
\]
where the supremum is taken over all metrics $g$
of non-positive sectional curvature.
For the class of Lagrangian submanifolds
$L$ under consideration we define
\[
a_{\mathrm m}^+(U)
:=\sup\{\inf(L)\,|\,
\text{$L\subset U$}\},
\]
cf.\ \cite[Theorem 2.5]{zeh12}.
The action-area inequality from Theorem \ref{aaineqthm} implies:

\begin{cor}
  For all open subsets $U$ of $\R^{2n}$ we have
  \[
  \ell_{\mathrm m}^+(U)\leq 2a_{\mathrm m}^+(U).
  \]
\end{cor}

With
\cite[Theorem 3.1]{zeh12}
we obtain
\[
\ell_{\mathrm m}^+(Z)=\pi,\qquad
\ell_{\mathrm m}^+(B)\in\left[\frac{\pi}{n},\frac{2\pi}{n}\right].
\]
Motivated by the work
of Cieliebak and Mohnke
on the Lagrangian capacity
\cite{cielmoh,chls07}
we \emph{conjecture}
that the link capacity
on the unit ball equals $\pi/n$.


\section{The relative spherical Gromov radius}
\label{relsphgromrad}

The aim of this section
is to prove Theorem \ref{finofsphervar}.


\subsection{Proof of Theorem \ref{finofsphervar}}
\label{prfofthm14}

Recall that a Lagrangian submanifold
$L$ in $\R^{2n}$ is monotone
if there exists a positive
real number $\eta$,
the so-called
{\bf monotonicity constant of} $L$,
such that the Liouville class
$\lambda_L$ and the Maslov
class $\mu_L$ of $L$
satisfy $\lambda_L=\eta\mu_L$.
The proof of Theorem \ref{finofsphervar}
below will show
that monotone Lagrangian tori satisfy
\[
s(L)\leq 4\eta.
\]
Notice that the minimal
positive symplectic area of a smooth
disc with boundary on $L$
is equal to $2\eta$.
Moreover,
it is attained by
a holomorphic disc with Maslov number $2$,
see \cite{buh10,damian12}.
A theorem of
Chekanov \cite{chek98}
implies that
\[
s(L)\leq 2d(L),
\]
where $d(L)$ denotes
the displacement energy
of $L$,
see \cite{hoze94}.
If in addition
there is a metric
of non-positive sectional
curvature on $L$
we get with \cite[Theorem 2.5]{zeh12}
\[
s(L)\leq \frac2k c_k^{\eh}(L)
\]
for all $k\in\N$
and the $k$-th Ekeland-Hofer
capacity \cite{ekho90}.

\begin{proof}[{\bf Proof of Theorem \ref{finofsphervar}}]
  The proof of the theorem is an application of the relative neck stretching argument
  due to Abbas \cite{abbasbook}.
  We consider a relative symplectic embedding $\varphi$ of $(S^{2n-1}_r,S^{n-1}_r)$ into $(\R^{2n},L)$.
  For $\varepsilon>0$ small enough we can assume that the neighbourhood
  on which $\varphi$ is defined
  contains the spherical shell
  \[
  U_{\varepsilon}=B_{r+\varepsilon}\setminus\overline{B_{r-\varepsilon}}.
  \]
  We consider the symplectic ellipsoid
  \[
  E=\Bigl\{\tfrac{x_1^2+y_1^2}{r_1^2}+\ldots+\tfrac{x_n^2+y_n^2}{r_n^2}<1\Bigr\}
  \]
  for real numbers
  \[
  r-\varepsilon<r_1<r_2<\ldots<r_n<r+\varepsilon
  \]
  and denote by $M=\varphi(\partial E)$
  the image of the boundary.
  Because $M$ is simply connected
  we find a global primitive $1$-form $\lambda$ of $\wst$
  such that 
  \[
  \lambda=\varphi_*\big(\lamst\big)
  \]
  on $\varphi(U_{\varepsilon})$.
  Notice that $\lambda$ vanishes on the tangent spaces of $L\cap\varphi(U_{\varepsilon})$.
  Denote by $\alpha$ the restriction of $\lambda$ to $TM$.
  Let $W$ and $V$ be the closures of the components of the complement of $M$
  such that
  \[
  \R^{2n}=W\cup_MV
  \]
  is a decomposition
  into a Liouville filling $(W,\lambda)$ of the contact type hypersurface $(M,\alpha)$
  and a symplectic manifold $(V,\rmd\lambda)$ with concave boundary $(M,\alpha)$.
  Observe,
  that the $(n-1)$-sphere $K=L\cap M$ bounds a domain inside the torus $L$
  such that one component $L\cap W$ or $L\cap V$ of the complement is simply connected.
  This follows with the Jordan-Schoenflies theorem \cite{brown60,mazur59, morse60}
  applied to the universal cover $\R^n$.
  
  We choose the radii $r_1,\ldots,r_n$ such that in addition
  there squares are rationally independent.
  Then all closed Reeb orbits on $(M,\alpha)$ correspond to intersections of $\partial E$
  with the complex coordinate axes
  and each Reeb chord of the Legendrian submanifold $K$
  is contained in a closed Reeb orbit.
  Moreover, the contact form $\alpha$ and the pair $(\alpha,K)$
  are generic in the sense of \cite[Chapter 3.2]{abbasbook},
  i.e.\ the linearized Poincar\'e return map restricted to the contact structure $\ker(\alpha)$
  at any periodic point of the Reeb flow has no eigenvalue $1$,
  and
  whenever the isotopic image $K'$ of $K$ under the Reeb flow intersects $K$ itself,
  the contact structure $\ker(\alpha)$ is spanned by the tangent spaces to $K$ and $K'$ at the intersection points.
  Therefore, the genericity assumptions of the compactness theorem in \cite{abbasbook}
  are satisfied.
  
  In order to define
  a sequence of almost complex
  structures on $\R^{2n}$
  we describe a symplectic neighbourhood
  $U\subset\varphi(U_{\varepsilon})$ of $M$.
  Let $Y$ be the Liouville
  vector field dual to $\lambda$.
  Following its flow near $M$
  in forward and backward
  time we obtain a symplectomorphic model
  \[
  \big([-\delta,\delta]\times M,\rmd(\rme^s\alpha)\big)
  \]
  of $U$ for $\delta>0$,
  see \cite{geig08},
  which we call the {\bf neck}.
  Notice that $Y$,
  which is mapped to $\partial_s$,
  is tangent to $L\cap U$
  so that the intersection of the Lagrangian submanifold $L$ with $U$
  corresponds precisely to $[-\delta,\delta]\times K$.
  In the same way we obtain for each $N\in\N$ a symplectomorphic copy of the neck
  \[
  \big([-N-\delta,N+\delta]\times M,\rmd(\tau\alpha)\big)
  \]
  with Lagrangian submanifold $[-N-\delta,N+\delta]\times K$,
  where $\tau$ is a smooth strictly increasing function on $[-N-\delta,N+\delta]$,
  which equals $\rme^{s+N}$ on $[-N-\delta,-N-\delta/2]$ and
  $\rme^{s-N}$ on $[N+\delta/2,N+\delta]$,
  see \cite[p.\ 158]{howyze03}.
  We define a translation invariant almost complex structure on the $N$-neck as follows:
  On $\ker(\alpha)$ it is required to restrict to a complex structure compatible with $\rmd\alpha$
  and on its complement it is required to map $\partial_s$ to
  the Reeb vector field of $\alpha$.
  This defines an almost complex structure $J_N$ on $U$ for each $N\in\N$.
  Near the boundary of $U$ it is independent of $N$
  and therefore can be extended to $\R^{2n}$ in a uniform way, see \cite{grom85}.
  Moreover,
  $J_N$ equals the complex structure of $\C^n$
  outside a fixed large ball.
  
  We consider the case where $L\cap W$ is simply connected.
  \emph{The case of $L\cap V$ being simply connected can not occur
  because otherwise there would be a Reeb chord on $K$ with negative action
  by the analogue of the following argument:}
  Let $p$ be a point in the interior of $L\cap W$.
  Because the Lagrangian torus $L$ is monotone
  we can apply Damian's result \cite[Theorem 1.5.(c)]{damian12}.
  Therefore,
  we find for any $N$ a $J_N$-holomorphic disc $u_N\co D\ra\R^{2n}$ through $p$ with boundary on $L$
  and Maslov index $2$.
  In particular,
  the energy of $u_N$ is
  \[
  \int_Du_N^*\rmd\lambda=2\eta
  \]
  for all $N$,
  where $\eta$ is the monotonicity constant of $L$.
  Notice,
  that the boundary curves $u_N(\partial D)\subset L$
  are not entirely contained in $L\cap W$
  because these are not contractible in $L$.
  
  By Abbas's compactness theorem \cite{abbasbook} a subsequence of $u_N$
  converges to a holomorphic building of total Hofer-energy equal to $2\eta$.
  Its level structure corresponds to $W\cup\big([0,\infty)\times M\big)$,
  several (or non) copies of $\R\times M$,
  and $\big((-\infty,0]\times M\big)\cup V$.
  The boundary of the building lies in
  $(L\cap W)\cup\big([0,\infty)\times K\big)$,
  the corresponding copies of $\R\times K$,
  and $\big((-\infty,0]\times K\big)\cup (L\cap V)$
  resp.,
  cf. \cite{behwz03,howyze03}.
  Moreover,
  the building consists of punctured holomorphic spheres
  and discs with Lagrangian boundary conditions
  such that at least one disc is contained in each level.
  Over the interior punctures these are asymptotic to closed Reeb orbits of $\alpha$
  and over boundary punctures to non-constant Reeb chords of $(\alpha,K)$,
  see \cite{abb99,hof93}.
  In particular, in $W\cup\big([0,\infty)\times M\big)$ there exists a finite energy disc $u$
  with at least one boundary puncture.
  Its Hofer-energy satisfies
  \[
  E(u)=\sup_{\tau}\int_{D\setminus\Gamma}u^*\rmd\lambda_{\tau}<2\eta,
  \]
  where $\Gamma\subset D$ is the set of punctures of $u$.
  The supremum is taken over all smooth (strictly) increasing functions $\tau$ on $[-\delta,\infty)$
  which equal $\rme^s$ on $[-\delta,-\delta/2]$ and converge to $1$ as $s\ra\infty$.
  The $1$-form $\lambda_{\tau}$ is given by $\lambda$ on $W\setminus\big([-\delta,0]\times M\big)$
  and by $\tau\alpha$ on $[-\delta,\infty)\times M$.
  Notice that $\lambda_{\tau}$ vanishes on vectors tangent to $[-\delta,\infty)\times K$.
  The boundary curves $u(\partial D\setminus\Gamma)$ can be homotoped into
  $[-\delta,\infty)\times K$ relative neighbourhoods of the punctures $\Gamma$
  because $L\cap W$ is simply connected.
  An application of Stokes's theorem
  (taking the asymptotic and the boundary conditions into account)
  yields
  that the sum of all periods of Reeb orbits over (the possibly empty set of)
  interior punctures of $u$
  and of the longueurs (i.e.\ the actions w.r.t.\ $\alpha$) of all Reeb chords
  over (the non-empty set of) boundary punctures of $u$
  is strictly less than $2\eta$.
  Therefore, we get
  \[
  \frac{\pi}{2}(r-\varepsilon)^2<2\eta,
  \]
  because the left hand side is precisely the shortest length of a Reeb chord in $\partial E$
  starting and ending on $\partial E\cap\R^n$.
  Letting $\varepsilon$ tend to zero this proves the theorem.
\end{proof}

\begin{rem}
  \label{genofarg}
  Our proof requires the existence of a holomorphic discs $D$ through any given point on a Lagrangian submanifold $L$
  with $\partial D\subset L$ for any admissible almost complex structure
  such that the energy of $D$ is uniformly bounded and $\partial D$ is not contractible in $L$.
  Theorem \ref{finofsphervar} generalizes accordingly.
  By \cite[Theorem 3.3.(b)]{damian12} this is the case
  if $L$ is a Lagrangian submanifold of a Liouville symplectic manifold $(X,\lambda)$ convex at infinity
  such that any compact set in $(X,\rmd\lambda)$ is displaceable.
  $L$ itself is required to be closed, oriented, and monotone
  such that the total singular $\Z_2$-homology of the universal cover $\widetilde{L}$ has finite dimension over $\Z_2$
  and the $\Z_2$-Euler characteristic of $\widetilde{L}$ does not vanish.
  The resulting discs in this situation all have Maslov index $2$.
  
  Moreover, we used in the proof
  that any hypersurface of contact type symplectomorphic to the sphere separates $X$
  and that any smoothly embedded separating $(n-1)$-sphere
  in $L$ bounds a simply connected domain in $L$.
  Notice that the $2$-torus $L$ has this property.
  Moreover,
  any manifold $L$
  such that any smoothly embedded $(n-1)$-sphere in $L$ bounds a homeomorphic $n$-disc
  satisfies this too.
  Examples can be obtained with the Jordan-Schoenflies theorem \cite{brown60,mazur59, morse60}
  if the universal cover is $\R^n$
  with $n\geq3$.  
  Therefore, the inequality $s(L)\leq 2d(L)$ holds in the situation described.
\end{rem}


\subsection{A remark on dimension $4$\label{rem4}}

In the case of a Stein surface $X$ and a monotone Lagrangian $2$-torus $L$ both quantities $c_B(L)$ and $s(L)$ coincide:
Consider a symplectic embedding $\varphi$ of $S^3_r$ into $X$ relative $L$.
Then $\varphi(S^3_r)$ cuts an exact symplectic filling of out $X$.
By a theorem of Gromov \cite{grom85} $\varphi$ extends
(after restriction to a smaller neighbourhood of $S^3_r$)
to the ball $B^4_r$,
see \cite[Theorem 9.4.2]{mcsa04}.
It follows from the proof of Theorem \ref{finofsphervar}
that the intersection of $\varphi(B^4_r)$ with $L$ is a $2$-disc.
Therefore, $\varphi^{-1}(L)$ is a local Lagrangian knot inside the ball $B^4_r$
in the sense of \cite{elipol96}.
As Polterovich pointed out to the author
with \cite[Theorem 1.1.A, Proposition 5.1.A.2)]{elipol96} one can assume
that the local Lagrangian knot is isotopic to $\R^2$ through local Lagrangian knots
whose trace of the non-flat regions stay inside $B^4_r$.
Hence, with Hamiltonian isotopy extension,
cf.\ \cite[p.\ 43]{polt01} and \cite[p.\ 96]{mcsa98},
$\varphi$ can be Hamiltonianly isotoped to $\psi$ inside $B^4_r$
such that $\psi$ coincides with $\varphi$ near $\partial B^4_r$
and maps $\R^2\cap B^4_r$ to $L$.
Hence, $\varphi$ extends to a symplectic embedding of $B^4_r$ relative $L$.


\subsection{A non-monotone example\label{nonmono}}

We consider a closed Lagrangian torus
\[
L=L'\times S^1_{\varrho}
\]
in $\R^{2n-2}\times\R^2$
such that the Lagrangian torus $L'\subset\R^{2n-2}$ is rational.
This means that the minimal positive symplectic area $\inf(L')$ of a disc with boundary on $L'$ is positive.
We choose the radius $\varrho$ of the circle $S^1_{\varrho}=\partial D_{\varrho}$
such that
\[
\frac{\inf(L')}{\varrho}
\]
is a natural number.
This implies that $L$ itself is rational with $\inf(L)=\varrho$.

Notice
that for the complex structure
of $\C^n$
the family
\[
\{*\}\times D_{\varrho}\subset L'\times\R^2
\]
of holomorphic discs defines a smooth filling.
With transversallity as in \cite{mcsa04,zehm08}
and Gromov compactness, see \cite{fra08,grom85,lazz00},
we get
that for all tamed almost complex structures $J$ standard at infinity
and any point $p$ on $L$
there exists a $J$-holomorphic map
\[
u\co (D,\partial D,1)\lra (\R^{2n},L,p)
\]
with symplectic area
\[
\int_{D}u^*(\wst)=\pi\varrho^2,
\]
see \cite[2.3.D.]{grom85}.
In view of Remark \ref{genofarg} we get:

\begin{cor}
  The relative spherical Gromov radius of the rational Lagrangian torus $L$ described above
  satisfies
  \[
  s(L)\leq 2\pi\varrho^2.
  \]
\end{cor}


\begin{ack}
  The main part of this work was carried out at the Universit\"at Hamburg.
  I would like to thank the Department of Mathematics for the hospitably.
  I thank
  Peter Albers, Fr\'ed\'eric Bourgeois, Urs Frauenfelder, Stefan Friedl,
  Hansj\"org Geiges, Janko Latschev, Leonid Polterovich, Felix Schlenk, and Stefan Suhr
  for their comments on the first version of these notes.
  Further I would like to thank Kai Cieliebak and Klaus Mohnke
  for sending me their preprint \cite{cielmoh}
  and Casim Abbas for providing me with the manuscript of his book \cite{abbasbook}.
\end{ack}



\begin{thebibliography}{10}

\bibitem{abb99}
  {\sc C. Abbas},
  Finite energy surfaces and the chord problem,
  \emph{Duke Math. J.} {\bf 96} (1999),
  241--316.

\bibitem{abbasbook}
  {\sc C. Abbas},
  \emph{Introduction to compactness results in symplectic field theory},
  Springer (2013), to appear

\bibitem{arn86}
  {\sc V. I. Arnol'd},
  The first steps of symplectic topology,
  \emph{Uspekhi Mat. Nauk} {\bf 41} (1986),
  3--18, 229.

\bibitem{artost08}
  {\sc S. Artstein-Avidan, Y. Ostrover},
  A {B}runn-{M}inkowski inequality for symplectic capacities of
  convex domains,
  \emph{Int. Math. Res. Not. IMRN} {\bf 13} (2008),
  Art. ID rnn044, 31.

\bibitem{barcor06} 
  {\sc J.-F. Barraud, O. Cornea},
  Homotopic dynamics in symplectic topology,
  in \emph{Morse theoretic methods in nonlinear analysis and in symplectic topology},
  109--148,
  NATO Sci. Ser. II Math. Phys. Chem., {\bf 217},
  Springer, Dordrecht, 2006.

\bibitem{barcor07}
  {\sc J.-F. Barraud, O. Cornea},
  Lagrangian intersections and the {S}erre spectral sequence,
  \emph{Ann. of Math. (2)} {\bf 166} (2007),
  657--722.

\bibitem{bates98}
  {\sc S. M. Bates},
  A capacity representation theorem for some non-convex domains,
  \emph{Math. Z.} {\bf 227} (1998),
  571--581.

\bibitem{bir97}
  {\sc P. Biran},
  Symplectic packing in dimension {$4$},
  \emph{Geom. Funct. Anal.} {\bf 7} (1997),
  420--437.

\bibitem{bircor09}
  {\sc P. Biran, O. Cornea},
  A {L}agrangian quantum homology,
  in: \emph{New perspectives and challenges in symplectic field theory},
  CRM Proc. Lecture Notes,
  {\bf 49},
  Amer. Math. Soc., Providence, RI, (2009),
  1--44.

\bibitem{bircor09b}
  {\sc P. Biran, O. Cornea},
  Rigidity and uniruling for {L}agrangian submanifolds,
  \emph{Geom. Topol.} {\bf 13} (2009),
  2881--2989.

\bibitem{behwz03}
  {\sc F. Bourgeois, Y. Eliashberg, H. Hofer, K. Wysocki, E. Zehnder},
  Compactness results in symplectic field theory,
  \emph{Geom. Topol.} {\bf 7} (2003),
  799--888.

\bibitem{brown60} 
  {\sc M. Brown},
  A proof of the generalized {S}choenflies theorem,
  \emph{Bull. Amer. Math. Soc.} {\bf 66} (1960),
  74--76.

\bibitem{buh10}
  {\sc L. Buhovsky},
  The {M}aslov class of {L}agrangian tori and quantum products in {F}loer cohomology,
  \emph{J. Topol. Anal.} {\bf2} (2010),
  57--75.

\bibitem{chek98}
  {\sc Yu. V. Chekanov},
  Lagrangian intersections, symplectic energy, and areas of holomorphic curves,
  \emph{Duke Math. J.} {\bf 95} (1998),
  213--226.

\bibitem{cieleli12}
  {\sc K. Cieliebak, Y. Eliashberg},
  \emph{From Stein to Weinstein and Back: Symplectic Geometry of Affine Complex Manifolds},
  Amer. Math. Soc. Colloq. Publ. {\bf 59},
  American Mathematical Society, Providence, RI (2012).

\bibitem{chls07}
  {\sc K. Cieliebak, H. Hofer, J. Latschev, F. Schlenk},
  Quantitative symplectic geometry,
  in: \emph{Dynamics, ergodic theory, and geometry},
  Math. Sci. Res. Inst. Publ.,
  {\bf 54},
  Cambridge Univ. Press, Cambridge (2007),
  1--44.
  
\bibitem{cielmoh}
  {\sc K. Cieliebak, K. Mohnke},
  Punctured holomorphic curves and Lagrangian embeddings,
  in preparation
  
\bibitem{damian12}
  {\sc M. Damian},
  Floer homology on the universal cover, Audin’s conjecture and other constraints on Lagrangian submanifolds,
  \emph{Comment. Math. Helv.} {\bf 87} (2012),
  433--462.

\bibitem{ekho90}
  {\sc I. Ekeland, H. Hofer},
  Symplectic topology and {H}amiltonian dynamics. {II},
  \emph{Math. Z.} {\bf 203} (1990),
  553--567.

\bibitem{elipol96}
  {\sc Y. Eliashberg, L. Polterovich},
  Local {L}agrangian {$2$}-knots are trivial,
  \emph{Ann. of Math. (2)} {\bf 144} (1996),
  61--76.
  
\bibitem{fra08}
  {\sc U. Frauenfelder},
  Gromov convergence of pseudoholomorphic disks,
  \emph{J. Fixed Point Theory Appl.} {\bf 3} (2008),
  215--271.

\bibitem{geig08}
  {\sc H. Geiges},
  \emph{An Introduction to Contact Topology},
  Cambridge Stud. Adv. Math. {\bf 109},
  Cambridge University Press, Cambridge (2008).
  
\bibitem{geizeh10}
  {\sc H. Geiges, K. Zehmisch},
  Eliashberg's proof of {C}erf's theorem,
  \emph{J. Topol. Anal.} {\bf 2} (2010),
  543--579.

\bibitem{geizeh12}
  {\sc H. Geiges, K. Zehmisch},
  Symplectic cobordisms and the strong Weinstein conjecture,
  \emph{Math. Proc. Cambridge Philos. Soc.} {\bf 153} (2012),
  261--279.

\bibitem{geizeh13}
  {\sc H. Geiges, K. Zehmisch},
  How to recognise a 4-ball when you see one,
  to appear in \emph{M\"unster J. Math.} (2013).
  
\bibitem{grom85}
  {\sc M. Gromov},
  Pseudoholomorphic curves in symplectic manifolds,
  \emph{Invent. Math.} {\bf 82} (1985),
  307--347.
  
\bibitem{herm04}
  {\sc D. Hermann},
  Inner and outer {H}amiltonian capacities,
  \emph{Bull. Soc. Math. France}
  {\bf 132} (2004), 509--541.

\bibitem{hof93}
  {\sc H. Hofer},
  Pseudoholomorphic curves in symplectizations with applications to the {W}einstein conjecture in dimension three,
  \emph{Invent. Math.} {\bf 114} (1993),
  515--563.

\bibitem{hls97}
  {\sc H. Hofer, V. Lizan, J.-C. Sikorav},
  On genericity for holomorphic curves in four-dimensional almost-complex manifolds,
  \emph{J. Geom. Anal.} {\bf 7} (1997),
  149--159.

\bibitem{howyze98}
  {\sc H. Hofer, K. Wysocki, E. Zehnder},
  The dynamics on three-dimensional strictly convex energy surfaces,
  \emph{Ann. of Math. (2)} {\bf 148} (1998),
  197--289.

\bibitem{howyze03}
  {\sc H. Hofer, K. Wysocki, E. Zehnder},
  Finite energy foliations of tight three-spheres and {H}amiltonian dynamics,
  \emph{Ann. of Math. (2)} {\bf 157} (2003),
  125--255.

\bibitem{hoze90}
  {\sc H. Hofer, E. Zehnder},
  A new capacity for symplectic manifolds,
  in: \emph{Analysis, et cetera},
  Academic Press, Boston, MA, (1990),
  405--427.

\bibitem{hoze94}
  {\sc H. Hofer, E. Zehnder},
  \emph{Symplectic invariants and {H}amiltonian dynamics},
  Birkh\"auser Verlag, Basel, (1994).

\bibitem{jiang99}
  {\sc M.-Y. Jiang},
  An inequality for symplectic capacity,
  \emph{Bull. London Math. Soc.} {\bf 31} (1999),
  237--240.

\bibitem{lazz00}
  {\sc L. Lazzarini},
  Existence of a somewhere injective pseudo-holomorphic disc,
  \emph{Geom. Funct. Anal.} {\bf 10} (2000),
  pp. 829--862.

\bibitem{mazur59}
  {\sc B. Mazur},
  On embeddings of spheres,
  \emph{Bull. Amer. Math. Soc.} {\bf 65} (1959),
  59--65.

\bibitem{mcd09}
  {\sc D. McDuff},
  Symplectic embeddings and continued fractions: a survey,
  \emph{Jpn. J. Math.} {\bf 4} (2009),
  121--139.

\bibitem{mcpol94}
  {\sc D. McDuff, L. Polterovich},
  Symplectic packings and algebraic geometry,
  With an appendix by Yael Karshon,
  \emph{Invent. Math.} {\bf 115} (1994),
  405--434. 
  
\bibitem{mcsa98}
  {\sc D. McDuff, D. Salamon},
  \emph{Introduction to symplectic topology},
  Oxford Mathematical Monographs, Second,
  The Clarendon Press Oxford University Press, New York, (1998).
  
\bibitem{mcsa04}
  {\sc D. McDuff, D. Salamon},
  \emph{$J$-holomorphic Curves and Symplectic Topology},
  Amer. Math. Soc. Colloq. Publ. {\bf 52},
  American Mathematical Society, Providence, RI (2004).
  
\bibitem{mcsch12}
  {\sc D. McDuff, F. Schlenk},
  The embedding capacity of 4-dimensional symplectic ellipsoids,
  \emph{Ann. of Math. (2)} {\bf 175} (2012),
  1191--1282. 

\bibitem{morse60} 
  {\sc M. Morse},
  A reduction of the Schoenflies extension problem,
  \emph{Bull. Amer. Math. Soc.} {\bf 66} (1960),
  113--115.

\bibitem{polt01}
  {\sc L. Polterovich},
  \emph{The geometry of the group of symplectic diffeomorphisms},
  Lectures in Mathematics ETH Z\"urich,
  Birkh\"auser Verlag, Basel (2001).

\bibitem{schl05}
  {\sc F. Schlenk},
  \emph{Embedding problems in symplectic geometry},
  de Gruyter Expositions in Mathematics {\bf 40},
  Walter de Gruyter GmbH \& Co. KG, Berlin (2005).

\bibitem{swozil12}
  {\sc J. Swoboda, F. Ziltener},
  Coisotropic displacement and small subsets of a symplectic manifold,
  \emph{Math. Z.} {\bf 271} (2012),
  415--445.

\bibitem{swozil12b}
  {\sc J. Swoboda, F. Ziltener},
  A symplectically non-squeezable small set and the regular coisotropic capacity,
  (2012),
  preprint, 
  arXiv:1203.2395,
  to appear in \emph{J. Symplectic Geom.}

\bibitem{vit00}
  {\sc C. Viterbo},
  Metric and isoperimetric problems in symplectic geometry,
  \emph{J. Amer. Math. Soc.} {\bf 13} (2000),
  411--431.

\bibitem{wein71}
  {\sc A. Weinstein},
  Symplectic manifolds and their Lagrangian submanifolds,
  \emph{Advances in Math.} {\bf 6} (1971),
  329--346.

\bibitem{zehm08}
  {\sc K. Zehmisch},
  \emph{Singularities and self-intersections of holomorphic discs},
  Doktorarbeit,
  Universit\"at {L}eipzig (2008).

\bibitem{zeh12}
  {\sc K. Zehmisch},
  Lagrangian non-squeezing and a geometric inequality,
  (2012),
  preprint, arXiv:1209.3704

\bibitem{zehzil12}
  {\sc K. Zehmisch, F. Ziltener},
  Discontinuous capacities,
  (2012),
  preprint, arXiv:1208.6000

\end{thebibliography}
\end{document}